\documentclass[a4paper,12pt]{article}

\usepackage{amsmath,amscd,amssymb,amsthm,amssymb}

\usepackage{cite}
\usepackage{bibgerm}

\usepackage[super]{nth}

\usepackage[british, USenglish]{babel}
\usepackage{tikz-cd}

\usepackage{fullpage}

\usepackage{newpxtext}
\usepackage{newpxmath}

\usepackage{enumitem}

\newtheorem{thm}{Theorem}[section]
\newtheorem{prop}[thm]{Proposition}
\newtheorem{lem}[thm]{Lemma}
\newtheorem{cor}[thm]{Corollary}
\theoremstyle{definition}
\newtheorem{rem}[thm]{Remark}
\newtheorem{dfn}[thm]{Definition}
\newtheorem{exmpl}[thm]{Example}

\newtheorem{manualtheoreminner}{Theorem}
\newenvironment{manualtheorem}[1]{%
  \renewcommand\themanualtheoreminner{#1}%
  \manualtheoreminner\slshape
}{\endmanualtheoreminner}

\newtheorem{manualcorollaryinner}{Corollary}
\newenvironment{manualcorollary}[1]{%
  \renewcommand\themanualcorollaryinner{#1}%
  \manualcorollaryinner\slshape
}{\endmanualcorollaryinner}

\usepackage{adjustbox}

\newcommand{\SOg}{\textrm{\textbf{SO}}}
\newcommand{\SUg}{\textrm{\textbf{SU}}}
\newcommand{\Spg}{\textrm{\textbf{Sp}}}

\newcommand{\Hom}{\textbf{\textbf{Hom}}}

\newcommand{\HH}{\mathbb H}
\newcommand{\ZZ}{\mathbb Z}
\newcommand{\RR}{\mathbb R}
\newcommand{\CC}{\mathbb C}

\newcommand{\PP}{\mathbb P}
\newcommand{\QQ}{\mathbb Q}

\newcommand{\id}{\textrm{id}}
\newcommand{\coker}{\textrm{coker}\,}
\newcommand{\im}{\textrm{im}\,}
\newcommand{\Sq}{\textrm{Sq}}
\newcommand{\biwe}{\mathord{\adjustbox{valign=B,totalheight=.6\baselineskip}{$\bigwedge$}\kern-0.1em}}
\newcommand{\biu}{\mathord{\adjustbox{valign=B,totalheight=.6\baselineskip}{$\bigcup$}}}

\numberwithin{equation}{subsection}
\title{Vector bundles and cohomotopies of spin $5$-manifolds}
\author{Panagiotis Konstantis\footnote{Universität Stuttgart}}
\date{}

\begin{document}

\maketitle

\begin{abstract}
  The purpose of this paper is two-fold: On the one side we would like to close
  a gap on the classification of vector bundles over $5$-manifolds. Therefore
  it will be necessary to study quaternionic line bundles over $5$-manifolds
  which are in $1-1$ correspondence to elements in the cohomotopy group $\pi^4(M)=[M,S^4]$
  of $M$. From results in \cite{MR0022071, MR3084240} this groups fits into a short exact
  sequence, which splits into $H^4(M;\ZZ)\oplus\ZZ_2$ if $M$ is spin. The second intent is to
  provide a bordism theoretic splitting map for this short exact sequence, which will lead
  to a $\ZZ_2$-invariant for quaternionic line bundles. This invariant is related to the generalized
  Kervaire semi-characteristic of \cite{MR1938494}.
\end{abstract}
\section{Introduction}\label{S:Introduction}
The classification of isomorphism classes of vector bundles over a fixed manifold (or more general over
a CW-complex) in terms of computable invariants (e.g. by characteristic classes) is a classical and
everlasting problem in topology. In particular in low dimensions such
classifications are feasible.

Knowingly every real and complex line bundle is completely determined by its first Stiefel-Whitney
and its first Chern class respectively. One of the first classification results
 was acquired by \emph{Dold} and \emph{Whitney} in \cite{MR0123331} on $4$-complexes. And later
\emph{Woodward} classified oriented $n$-dimensional vector bundles over $n$-complexes for
$n=3,4,6,7,8$ in terms of characteristic classes by using elementary homotopy theoretic methods in
\cite{MR677482}.

The gap in Woodward's classification, namely the case $n=5$, appears to be somewhat different as
the example of $S^5$ shows: By the clutching construction the isomorphism classes of oriented rank
$5$ vector bundles over $S^5$ are enumerated by $\pi_4(\SOg(5))\cong \ZZ_2$ and are represented by
the trivial and the tangent bundle of $S^5$. Of course both vector bundles have trivial
characteristic classes and therefore they cannot be used to distinguish them. In \cite{MR1258434}
\emph{\v{C}adek} and \emph{Van\v{z}ura} classify oriented rank $5$ vector bundles over a $5$-complex $X$
provided the following assumptions are fulfilled (see \cite[p.\, 755]{MR1258434})
\begin{itemize}
  \item[\textbf{(A)}] $H^4(X;\ZZ)$ has no element of order $4$,
  \item[\textbf{(B)}] $\Sq^2\, H^3(X;\ZZ) = H^5(X;\ZZ_2)$.
\end{itemize}
An important example of a $5$-complex satisfying condition \textbf{(B)} is a
closed $5$-manifold $M$ with $w_2(M)\neq 0$ (since by Poincar\'{e} duality the Wu class of $\Sq^2$
is just $w_2(M)$). The authors obtain
\begin{thm}[Theorem 1, \cite{MR1258434}]\label{T:CadekVanzuraTheorem}
  Let $X$ be a CW-complex of dimension $\leq 5$ and suppose
  \[
    \gamma \colon [X, B \SOg(5)] \to H^2(M;\ZZ_2) \oplus H^4(M;\ZZ_2) \oplus H^4(M;\ZZ)
  \]
  is defined by
  \[
    \gamma(V) = (w_2(V),w_4(V),p_1(V)),
  \]
  where the triple consists of the second, fourth Stiefel-Whitney and the first Pontryagin class of
  $V$ respectively. Then
  \begin{enumerate}[label=(\alph*)]
    \item $\im\gamma = \{(a,b,c) : \rho_4(c) = \mathfrak{P}a + i_\ast b\}$ (where $\rho_4$ is the
      $\mod 4$ reduction in cohomology, $\mathfrak{P} \colon H^2(M;\ZZ_2) \to H^4(M;\ZZ_4)$ is
      a Pontryagin square (cf. \cite[Chapter 2]{MR0226634}) and $i_\ast\colon H^\ast(M;\ZZ_2)\to
      H^\ast(M;\ZZ_4)$ the homomorphism induced by the map $\ZZ_2 \to \ZZ_4$).
    \item  If conditions \textbf{(A)} and \textbf{(B)} are fulfilled then $\gamma$ is injective.
  \end{enumerate}
\end{thm}
Hereby we note that results in \cite{MR677482} are in similar fashion as Theorem
\ref{T:CadekVanzuraTheorem}, especially with conditions like \textbf{(A)}.

The paper in hand will fill the gap for $n=5$ where $M$ is a smooth, closed $5$-manifold with
$w_2(M)=0$. We call a vector bundle $V \to M$ \emph{spinnable} if $w_1(V)=w_2(V)=0$. We prove
in Proposition \ref{P:ReductionOfStructureGroup} that every spinnable vector bundle of rank
$5$ over $M$ is decomposed by $E \oplus \varepsilon^1$ where $E$ is a quaternionic line bundle.

Unlike the case of real and complex vector bundles the set of quaternionic line bundles do not
posses in general a group structure. However if $M$ is of dimension five, the quaternionic line
bundles over $M$ are in $1-1$ correspondence with elements of the cohomotopy group
$\pi^4(M)=[M,S^4]$ which has a natural group structure, see Remark \ref{R:GroupStructure}. In section
\ref{SS:Steenrods_computation_of_pi4M} we explain that, from previous works, $\pi^4(M)$
fits into an short exact sequence which splits such that $\pi^4(M)\cong H^4(M;\ZZ)\oplus\ZZ_2$ if
$M$ is spin.
The \emph{Pontryagin-Thom construction} provides an isomorphism between $\pi^4(M)$ and
$\Omega_{1;M}^{\text{fr}}$, the bordism classes of normally framed closed $1$-submanifolds of $M$.
In section \ref{SS:GeometricInterpretation} we assign to every quaternionic line bundle $E \to M$
a \emph{framed divisor} $[L,\varphi_E] \in \Omega_{1;M}^{\text{fr}}$ where $L$ is the zero locus
of a generic real section of $E$ and $\varphi_E$ a framing on the normal bundle of $L$ induced by
the $\Spg(1)$-structure of $E$. In section \ref{SS:SplittingMap} we construct an invariant
\[
  \kappa \colon \pi^4(M) \longrightarrow \Omega_1^{\text{fr}},
\]
where $\kappa(E)$ is the \emph{stabilized framed divisor} of $E$ and $\Omega_1^{\text{fr}}$
the bordism group of stably framed closed $1$-manifolds, see Definition \ref{D:generalizedKSC}. The
definition of $\kappa$ can depend on the choice of a spin structure on $M$.
The main result in this section is Theorem \ref{T:kappaIsAsection} where we show that the
$\ZZ_2$-part of $\pi^4(M)$ is isomorphic to $\pi_5(S^4)$ and furthermore that $\kappa$ is a section
for the short exact sequence of $\pi^4(M)$. We obtain
\begin{manualcorollary}{\ref{C:TheSplitting}}
Let $M$ be a closed spin $5$-manifold. Then for any $\Spg(1)$-structure on $M$ the map
\[
  \pi^4(M) \to H^4(M;\ZZ) \times \Omega_1^{\text{fr}},\quad E\mapsto \left(
  \frac{p_1}{2}(E),\kappa(E)  \right)
\]
  is an isomorphism of groups (where $\tfrac{p_1}{2}(E)$ is the spin characteristic class of $E$).
\end{manualcorollary}

Cohomotopy groups of a manifolds was also studied by \emph{Kirby, Teichner} and
\emph{Melvin} in \cite{MR3084237}  where the authors compute by elementary geometric arguments the
cohomotopy group $pi^3$ of $4$-manifolds. For $X$ an odd $4$-manifold the authors show in
\cite[Theorem 1]{MR3084237} that the short exact sequence of $\pi^3(X)$ splits if and only if $X$
is spin.

The classification of quaternionic line bundles over quaternionic projective spaces was studied in
\cite{MR2291170, MR0336740}. Our results show similarities to the work of \cite{MR3034463}.
\emph{Crowley} and \emph{Goette} introduce an index theoretic $t$-invariant (cf. \cite[Definition
1.4]{MR3034463}) for quaternionic line bundles $E$ on $(4k-1)$-manifolds $N$ with
$H^3(N;\QQ)=0$ and such that the spin characteristic class of $E$ is torsion. If $\text{Bun}(N)$
denotes the set of isomorphism classes of quaternionic line bundles, then the $t$-invariant
is a map $t \colon \text{Bun}(N) \to \QQ/\ZZ$ such that
\[
  \text{Bun}(N) \longrightarrow H^4(N;\ZZ) \times \QQ/\ZZ, \quad E \mapsto \left( \tfrac{p_1}{2}(E),
  t(E)\right)
\]
is injective provided $N$ is a smooth, $2$-connected oriented rational homology $7$-sphere, see
\cite[Theorem 0.1]{MR3034463} and cf. Corollary \ref{C:TheSplitting} (note that $\tfrac{p_1}{2}(E)
= -c_2(E)$, if we consider $E$ as a vector bundle with $\SUg(2)$-structure, see section
\ref{SS:GeometricInterpretation}).
Moreover the concept of a \emph{divisor} for quaternionic line bundles was also used in
\cite{MR3034463} to show that the $t$-invariant localizes near its divisor,
cf. \cite[Proposition 1.10]{MR3034463}. Finally the $t$-invariant and the $\kappa$-invariant indicate more
resemblances as \cite[Proposition 1.11]{MR3034463} shows: For a stably framed manifold $N$ one
obtains $t(E)=-e(Y)$, where $Y$ is a divisor of $E$ and $e$ Adam's $e$-invariant, cf.
\cite[Section 7]{MR0198470}. Thus $t$ is, as $\kappa$, related to the theory
of stable homotopy groups and the $J$-homomorphism.

In section \ref{S:Classification of spin vector bundles} we will use the developed theory of
quaternionic line bundles to classify spinnable vector bundles of rank $5$ over spin $5$-manifolds
in terms of $\kappa$ and the spin characteristic class $\tfrac{p_1}{2}$:

\begin{manualtheorem}{\ref{T:Classification of V}}
Let $M$ be a closed spin $5$-manifold and consider the sets
  \begin{align*}
    W_1 &:= \{ V \in [M,B \SOg(5)] : w_2(V)=w_4(V)=0\},\\
    W_2 &:= \{ V \in [M,B \SOg(5)] : w_2(V)=0,\,w_4(V)\neq 0\}.
  \end{align*}
Then $W_1$ is a group such that the map
\[
  W_1 \to  H^4(M;\ZZ) \oplus\Omega_1^{\text{fr}}, \quad V \mapsto \left( \tfrac{p_1}{2}(V), \kappa(V)  \right)
\]
  is an isomorphism of groups.

  Furthermore if $\dim H^4(M;\ZZ_2)>0$ then every element in $W_2$ is
  uniquely determined by its spin characteristic class $\tfrac{p_1}{2}$ which are in one-to-one
  correspondence to $\rho^{-1}_2(H^4(M,\ZZ_2)\setminus\{0\})$, where $\rho_2$ induced by the mod $2$ reduction
  $\ZZ \to \ZZ_2$. In particular every vector bundle in $W_2$ is determined by its stable class.
\end{manualtheorem}

Moreover for $V \in W_1$ we define $\kappa(V)$ to be $\kappa(E)$, where $E$ is the unique
quaternionic line bundle such that $V \cong E \oplus \varepsilon^1$. Proposition
\ref{P:Connection to generalized Semi Kervaire Charactersitic} shows that $\kappa(V)$ is the
\emph{generalized Kervaire semi-characteristic} of $V$, cf. \cite{MR1938494}. Thus we provide
a bordism theoretic definition of the generalized Kervaire semi-characteristic.

At the end we point out some implications for the topology of $5$-manifolds using that
generalized Kervaire semi-characteristic for tangent bundles is the \emph{Kervaire
semi-characteristic} which is defined as
\[
  k(M) = \sum_{i}^{} \dim H^{2i}(M;\RR) \mod 2,
\]
see section \ref{S:Classification of spin vector bundles}.

\paragraph*{Acknowledgments.} The author would like to thank Oliver Goertsches and Achim Krause
for helpful discussions. Parts of this work were written at the Philipps-Universität Marburg,
where the author has the privilege to be a guest in the research group \emph{Differential Geometry and
Analysis}. For that reason the author is thankful for Ilka Agricola's hospitality.
\section{Spinnable vector bundles}\label{S:Spinnable vector bundles}
Let $V \to M$ be an oriented, spinnable (i.e. $w_2(V)=0$) vector bundle over $M$ of rank $5$.
Before we prove a structure result for $V$ we have to discuss subgroups of $\SOg(n)$ and $\Spg(n)$
and their relations in low dimensions.

Let $i \colon \SUg(2) \to \SOg(5)$ denote the canonical inclusion $\SUg(2) \subset \SOg(4)\subset
\SOg(5)$. If we identify $\Spg(1)$ with $\SUg(2)$ we may consider $\Spg(1)$ as a subgroup of
$\SOg(5)$ by $i$. Furthermore we consider one of the possible two diagonal embeddings of $\Spg(1)$ into
$\Spg(2)$ which we call a \emph{standard embedding}. We remark, that we will not distinguish
between an $\Spg(1)$ and an $\SUg(2)$-structure.
\begin{lem}[see Section 2 of \cite{MR2120916}] \label{L:RepresentationTheory}
Let $i \colon \Spg(1) \to \SOg(5)$ be the inclusion described above. Then $i$ factors as the
 standard embedding of $\Spg(1)$ into $\Spg(2)$ through
 the universal cover map $\pi \colon \Spg(2) \to \SOg(5)$.
\end{lem}

We define a quaternionic line bundle by means of reductions of structure groups.
\begin{dfn}\label{D:QuaterionicLineBundle}
Let $E \to M$ be an oriented real vector bundle of rank $4$. We say $E$ is \emph{quaternionic line
  bundle} if the structure group of $E$ can be reduced to $\Spg(1)\cong \SUg(2)$.
\end{dfn}
Now the important structure result for $V$ is
\begin{prop}\label{P:ReductionOfStructureGroup}
There is a quaternionic line bundle $E \to M$ such that $V \cong E \oplus \varepsilon^1$.
\end{prop}
\begin{proof}
  Let $B_k$ denote the classifying space $B\SOg(k)$ and let $B_{\mathbb H} := B\SUg(2)\cong
  B\Spg(1)$. Then consider the fibration
  \[
    \SOg(5)/\SUg(2) \longrightarrow B_{\mathbb H} \overset{Bi}{\longrightarrow}  B_5,
  \]
  where $Bi$ is the map induced from $i \colon \Spg(1) \to \SOg(5)$ on the classifying spaces.
  Denote with the same letter $V \colon M \to B_5$ the classifying map of $V$. We will show that under the
  condition $w_2(V)=0$ there is a lift of $V$ to a map $M \to B_{\mathbb H}$ which will prove the
  lemma.

  Using Lemma \ref{L:RepresentationTheory}, the universal cover of $\SOg(5)/\SUg(2)$ is
  $\Spg(2)/\Spg(1)$, where $\Spg(1)$ lies in $\Spg(2)$ by a standard embedding of $\mathbb H$ in
  $\mathbb H^2$. Since $\Spg(2)/\Spg(1)=S^7$
  we obtain
  \[
    \pi_1(\SOg(5)/\SUg(2))=\ZZ_2,\quad \pi_k(\SOg(5)/\SUg(2))=0,\, \text{for } k=2,3,4.
  \]
  It follows that if $M$ is of dimension $5$, the only obstruction for lifting $V$ lies in $H^2(M;\ZZ_2)$
  and is given by $w_2(V)$, since the vanishing is a necessary condition.
\end{proof}
Thus to understand the spinnable rank $5$ vector bundles one has to understand first the set
of quaternionic line bundles over $M$.
\begin{rem}\label{R:GeometricSU(2)Structures}
The geometric properties of $\SUg(2)$-structures on $5$-manifolds were studied intensively in
\cite{MR2327032, MR2559628}.
\end{rem}
\section{Quaternionic vector bundles}\label{S:quaternionic_vector_bundles}
The classifying space $B\Spg(1)\cong B\SUg(2) \cong B \mathbf{Spin}(3)$ for quaternionic line
bundles is given by the infinite quaternionic projective space $\HH\PP^{\infty}$ and the inclusion
$S^4 \to \HH\PP^{\infty}$ is an $7$-equivalence. Thus if $M$ is a $5$-dimensional manifold the set
of isomorphism classes of quaternionic line bundles are in $1$-$1$ correspondence to
\[
  [M,S^4] = \pi^4(M),
\]
see also \cite{MR0029170}.
Hence, the set of quaternionic line bundles over $M$ possesses naturally the structure of group,
which is in general false in higher dimensions and contrary to the case of real and complex line
bundles. We conclude that every quaternionic vector bundle $E \to M$ is the pull-back of the
tautological quaternionic line bundle $H$ over $S^4 = \HH\PP^{1}$. We will mix notations and
denote a quaternionic line bundle over a $5$-manifold also by a homotopy class of a continuous map
$M \to S^4$.

\begin{rem}\label{R:GroupStructure}
  \begin{enumerate}[label=(\alph*)]
    \item Let us describe  the group structure of $\pi^4(M)$:
          Consider the inclusion $j \colon S^4 \vee S^4 \to S^4 \times S^4$ of the
          $7$-skeleton of $S^4\times S^4$ (endowed with the standard CW-structure). Since $M$ is
          $5$-dimensional, the induced map $j_\# \colon [M,S^4\vee S^4] \to [M,S^4\times S^4]$ is
          bijective. For $f,g \in \pi^4(M)$ the group structure is defined by \[ f + g := (\id_{S^4}
          \vee\id_{S^4})_\# \circ (j_\#)^{-1}(f\times g). \] This makes $\pi^4(M)$ into an abelian group.
        \item A more sophisticated view to the group structure can be found in
          \cite[§6]{MR3084240}: Let $SE_4$ be the first two stages of the Postnikov decomposition
          of $S^4$ then the map $S^4 \to SE_4$ induces a bijection $[M,S^4]\to [M,SE_4]$ (cf.
          \cite{MR0060234}). But we have $SE_4 \cong \Omega SE_5$, thus $SE_4$ is a homotopy $H$-space
          and therefore $[M,S^4]$ carries a natural group structure.
  \end{enumerate}
\end{rem}

\subsection{Steenrod's enumeration of $\pi^4(M)$}\label{SS:Steenrods_computation_of_pi4M}

  Steenrod investigated $\pi^{n}(X)$ in \cite[Theorem 28.1, p. 318]{MR0022071} for a CW complex $X$
  of dimension $\dim X =n+1$. If $\sigma \in H^n(S^n;\ZZ)$ is a generator then the
  \emph{Hurewicz-homomorphism}
  \[
    \Phi \colon \pi^{n}(X) \to H^n(X;\ZZ), \quad f \mapsto f^*(\sigma)
  \]
  is a surjective group homomorphism and the kernel is isomorphic to
  \[
    H^{n+1}(X;\ZZ_2)/\Sq^2\,\mu(H^{n-1}(X;\ZZ))
  \]
  where $\mu\colon H^i(X;\ZZ) \to H^i(X;\ZZ_2)$ is the map induced by the coefficient
  homomorphism $\ZZ \to \ZZ_2$. Thus, if $M$ is spin, the \emph{Wu class} of $\Sq^2$
  is $w_2(M)$ and therefore  we obtain a central extension of $\ZZ_2$
      \begin{equation}
        \label{eq:SES}
        0 \longrightarrow \ZZ_2 \longrightarrow \pi^4(M) \longrightarrow H^4(M;\ZZ)
        \longrightarrow 0
      \end{equation}
  Moreover in \cite[§6]{MR3084240} Taylor uses  methods of Larmore and Thomas \cite{MR0328935}
  to study this extension. A purely homotopy theoretical argument shows that the above short
  exact sequence splits if $\Sq^2 \colon H^3(M;\ZZ) \to H^{5}(M;\ZZ_2)$ and  $\Sq^2 \colon
  H^3(M;\ZZ_2) \to H^{5}(M;\ZZ_2)$ have the same image, cf. \cite[Example 6.3]{MR3084240}. This is
  obviously the case when $M$ is spin.  We would like to give a geometrically meaning to
   splitting $\pi^4(M) = H^4(M;\ZZ) \oplus \ZZ_2$.

   Before doing this we would like to explain that $\Phi$ is equal to the map which assigns
   to every quaternionic line bundle its spin characteristic class.
Let $\SUg = \biu_n \SUg(n)$ and $\mathbf{Spin} = \biu_n \mathbf{Spin}(n)$. The canonical
inclusion $\SUg \to \mathbf{Spin}$ induces isomorphisms $\pi_i\, \SUg \to \pi_i\, \mathbf{Spin}$
for $i\leq 5$, see \cite[Lemma 2.4]{MR2442894}. Define the \emph{ (universal) spin characteristic class}
$\tfrac{p_1}{2} \in H^4(B \mathbf{Spin};\ZZ)$ to be the preimage of the universal
second Chern class $-c_2 \in H^4(B\SUg;\ZZ)$ under the map $B\SUg \to B \mathbf{Spin}$ induced
by the canonical inclusion.

Let $W \to X$ be a spinnable vector bundle over a finite CW complex $X$. A choice of spin
structure on $W$ defines a map $g: X \to B \mathbf{Spin}$ which is a lift of the classifying map $X \to
\SOg = \biu _n \SOg(n)$. Define $\tfrac{p_1}{2}(W)$ to be $g^\ast(\tfrac{p_1}{2}) \in H^4(X;\ZZ)$,
which is independent of the choice of spin structure, cf. \cite[p.\,170]{MR2442894}. Moreover
we have
  \begin{equation}
    \label{eq:ReductionsOfSpinClasses}
    \tfrac{p_1}{2}(W) \equiv w_4(M) \mod 2,\quad
    2 \cdot \tfrac{p_1}{2}(W) = p_1(W),
  \end{equation}
see again \cite[p.\,170]{MR2442894}.

The classifying map
\[
  \frac{p_1}{2} \colon B \mathbf{Spin} \longrightarrow K(Z,4)
\]
is an $8$-equivalence, thus two spinnable vector bundles over a CW-complex of dimension $\leq 7$
are stably isomorphic if and only if their spin characteristic class $\tfrac{p_1}{2}$ are equal, cf.
\cite{MR677482} and \cite[p. 5]{1603.09700v2}.

The inclusion $i \colon S^4 = \HH\PP^{1} \hookrightarrow \HH\PP^{\infty}$ induces an isomorphism
on integer cohomology in dimension $4$. By construction $\tfrac{p_1}{2}(H) \in H^4(S^4;\ZZ)$ is
a generator where $H$ can be described as the pull back under $i$ of the tautological quaternionic
line bundle over $\HH\PP^{\infty}$ . Thus the map $\Phi$ is given as
\[
  \Phi \colon \pi^4(X) \to H^4(X;\ZZ),\quad f\mapsto \tfrac{p_1}{2}(f^\ast (H)) =
  f^\ast\left( \tfrac{p_1}{2}(H) \right).
\]
\subsection{Geometric interpretation of $\pi^4(M)\cong
H^4(M;\ZZ)\oplus \ZZ_2$.}\label{SS:GeometricInterpretation}

Suppose now $E \to M$ is a quaternionic line bundle over a spin $5$-manifold $M$. We would like
to show how to obtain element of $\pi^4(M)$ out of the bundle data of $E$. Therefore we need
to make a detour over $\Omega_{1;M}^\text{fr}$, the bordism group of normally framed closed
$1$-manifolds of $M$ (see \cite[§7]{MR0226651} for a definition). The groups $\pi^4(M)$ and
$\Omega_{1;M}^{\text{fr}}$ are isomorphic by the \emph{Pontryagin-Thom construction}, see again
\cite[§7]{MR0226651}.

Choose a (real) section $\sigma  \in \Gamma(E)$ transverse to the zero section of $E$. The zero set
of $\sigma$, call it $L$, is a closed $1$-dimensional submanifold of $M$. The bundle $E$ defines
a trivialization of the real vector bundle $\nu_L:=\nu(L \hookrightarrow M)$ as follows. The
restricted vector bundle $E|_L$ is the trivial quaternionic line bundle and any non vanishing
quaternionic section of $E|_L$ gives a trivialization $E|_L \cong L \times \mathbb H$
which is unique up to homotopy (since $\pi_1(\Spg(1))=1$). The quaternionic structure gives now, up
to homotopy, a unique trivialization of the underlying real vector space $E|_L$. The
derivative of $\sigma$ along $L$ gives a bundle isomorphism between $\nu_L$ and
$E|_L$. Let us denote this induced framing of $\nu_L$ with $\varphi_E$. Thus we obtain
a class $[L,\varphi_E] \in \Omega_{1;M}^{\text{fr}}$.

\begin{lem}\label{L:FramedDivisorIsWellDefined}
  The class $[L,\varphi_E]$ does not depend on the choice of the transverse section $\sigma$.
\end{lem}
\begin{proof}
Let $\sigma'$ be another section of $E$ transverse to the zero section and denote the set
of zeros by $L'$. These data induce, like above, a framing $\varphi'_E$ on $\nu_{L'}$. Then there
is a section $\tau \in \Gamma(\tilde E)$ such that $\tau|_{M\times 0} = \sigma$ and $\tau|_{M\times
1} = \sigma'$ where $\tilde E = \textrm{pr}^\ast(E)$ and $\textrm{pr} \colon M\times [0,1] \to M$
is the projection onto the first factor. We may also assume that $\tau$ is transverse to the zero
  section of $\textrm{pr}^\ast(E)$ and $\tau(p,t)=\sigma(p)$, $\tau(p,1-t)=\tilde \sigma(p)$ for
  $t \in [0,\varepsilon)$ where $\varepsilon>0$ is small. The
  zero set $\Sigma$ of $\tau$ is a two-dimensional submanifold with boundary $(L\times 0) \cup
  (L'\times
  1)$. The section $\tau$ provides an isomorphism of $\nu( \Sigma \hookrightarrow M \times[0,1])$
    and $\tilde E|_{\Sigma}$. And as before $\tilde E|_{\Sigma}$ is the trivial bundle with
    a canonical framing induced by the quaternionic structure (note that $\Sigma$ is a CW-complex
    of dimension at most $1$).

  Thus $\Sigma$ has a normal framing which restricts on the boundary to the given ones (by
  construction), which means, that $(L,\varphi_E)$ and $(L',\varphi'_E)$ are normally framed
  bordant, hence they represent the same element in $\Omega_{1;M}^{\text{fr}}$.
\end{proof}
\begin{dfn}\label{D:FramedDivisor}
  For a quaternionic line bundle $E\to M$ we call the class $[L,\varphi_E]$  the \emph{framed
  divisor of E}.
\end{dfn}
\begin{lem}\label{L:Depends only on the isomorphism class}
Two isomorphic quaternionic line bundles have the same framed divisor.
\end{lem}
\begin{proof}
The proof follows the same arguments as Lemma \ref{L:FramedDivisorIsWellDefined}, therefore
we sketch a proof. If $E$ and $E'$ are isomorphic quaternionic line bundles then we have a
  quaternionic line bundle over $M \times [0,1]$ which restricts to $E$ and $E'$ on $M\times 0$
  and $M \times 1$ respectively. Let $\sigma$ and $\sigma'$ be two sections of $E$ and $E'$
  respectively, then there is a section $\tau$ of the quaternionic line bundle
  over $M \times [0,1]$ which restricts to the given ones and is transverse to the zero section.
  The zero locus $\Sigma$ of $\tau$ provides a normally framed bordism in $M \times [0,1]$
  between the two framed divisors of $E$ and $E'$.
\end{proof}
This shows that we obtain a well-defined map from $\text{Bun}(M) \to \Omega_{1,M}^{\text{fr}}$,
$E\mapsto [L,\varphi_E]$, where $\text{Bun}(M)$ denotes the set of isomorphism classes of
quaternionic line bundles over $M$.
\begin{lem}\label{L:BundleAndCohomotopy}
  Let $f \colon M \to S^4$ be a classifying map for $E$. Under the Pontryagin-Thom isomorphism
  $\Omega_{1;M}^{\text{fr}} \to \pi^4(M)$ the framed divisor $[L,\varphi_E]$ is mapped to $f$.
\end{lem}
  \begin{proof}
    Suppose $x_0 \in S^4$ is a regular value of $f$ and denote by $(f^{-1}(x_0),\varphi_f)$ the
    Pontryagin-manifold to $f$. Then there is a section $\sigma_0 \colon M \to H$ with only one
    zero exactly at $x_0$ and transverse to the zero section. Observe now, that the pull back
    $\sigma:=f^\ast(\sigma_0)$ in $E=f^\ast(H)$ is transverse to the zero section and has
    $f^{-1}(x_0)$ as its set of zeros. Choosing a non-zero element in $H_{x_0}$ corresponds
    to a quaternionic section of $E|_{f^{-1}(x_0)}$ and the framings $\varphi_E$ and $\varphi_f$
    coincide.
  \end{proof}
\begin{rem}\label{R:GroupStructureOnBun_0}
  Thus we showed that the following diagram commutes
  \[
  \begin{tikzcd}
    & \text{Bun}(M) \arrow[rd, "{E\mapsto [L,\varphi_E]}"] & \\
    \pi^4(M) \arrow[ru, "f\mapsto f^\ast(H)"] & & \arrow[ll, "\text{Pontryagin-Thom}"]
     \Omega_{1;M}^{\text{fr}}.
  \end{tikzcd}
  \]
  We define on $\text{Bun}(M)$ a group structure such that the above diagramm commutes as abelian
  groups.
\end{rem}
%
\subsection{Splitting map}\label{SS:SplittingMap}
We will construct in this paragraph a splitting map for the short exact sequence \eqref{eq:SES}
Therefore we choose first a $\Spg(1)$-structure on $TM$ (which is equivalent to choose a spin
structure for $M$, cf. Proposition \ref{P:ReductionOfStructureGroup}). This is possible since we
assume $w_2(M)=0$. If $L$ is a closed $1$-dimensional submanifold of $M$ then $TM|_L$ is canonical
trivialized by the $\Spg(1)$-structure of $TM$.

Recall that if $E \in \pi^4(M)$ and $\sigma \in \Gamma(E)$ is a transverse section then the set of
zeros $L$ is a $1$-dimensional closed submanifold. Let $L = L_1 \cup \ldots L_k$ be the connected
components. The section $\sigma$ induces on each $\nu_{L_i}$ a framing as described above.
Altogether this produces a stable framing of $TL_i$ since
\[
  \varepsilon^5 \cong TM|_{L_i} \cong TL_i \oplus \nu_{L_i} \cong  TL_i \oplus \varepsilon^4.
\]
We denote this stable framing by $[L_i,\varphi_{E,i}^S]$ which is an element in
$\Omega_1^{\text{fr}}\cong \ZZ_2$. This definition is well-defined on the isomorphism class of $E$
or rather on the bordism class of the framed circle $[L,\varphi_E] \in \Omega_{1;M}^{\text{fr}}$.

\begin{lem}\label{L:kappaIswellDefined}
  Let $E \to M$ be a quaternionic line bundle over $M$ and let $\sigma,\sigma' \in \Gamma(E)$ be
  two transverse sections. Denote by $L$ and $L'$ the corresponding zero loci. Then
  \[
    [L,\varphi_E^S] = [L',(\varphi_E')^S] \in \Omega_{1}^{\text{fr}}.
  \]
\end{lem}
\begin{proof}
  Let $\text{pr} \colon M \times [0,1] \to M$ be the obvious projection. We have seen above, that
  $[L,\varphi_E]$ and $[L,\varphi_E']$ are normally framed bordant. Denote by $\Sigma \subset
  M \times [0,1]$ the normally framed bordism. Then $\Sigma$ is also stably framed
  \[
    \varepsilon^6 \cong T(M\times [0,1])|_{\Sigma} \cong T\Sigma \oplus \nu(\Sigma \hookrightarrow
    M\times [0,1]) \cong T\Sigma \oplus \text{pr}^\ast(E)|_{\Sigma} \cong T\Sigma \oplus
    \varepsilon^4
  \]
  and clearly this shows that $[L,\varphi_E^S]$ and $[L',(\varphi_E')^S]$ are stably framed bordant.
\end{proof}

\begin{dfn}\label{D:generalizedKSC}
    For a quaternionic line bundle $E \to M$ we call $[L,\varphi_E^S]$ the \emph{stabilized framed divisor} of $E$
    and we define
  \[
    \kappa(E):= \sum_{i=1}^{k}\, [L_i,\varphi_{E,i}^S] \in \Omega_1^{\text{fr}}
  \]
  where $L_i$ is connected and $L= L_1\cup \ldots \cup L_k$ is the zero locus of a transverse section of
  $E$.
\end{dfn}
\begin{cor}\label{C:IvanricanceUnderIsomorphism}
  If $E \to M$ and $E' \to M$ are two isomorphic quaternionic line bundles then
    $\kappa(E) = \kappa(E')$.
\end{cor}
\begin{exmpl}
  \label{E:5Sphere}
  Consider $S^5 \subset \CC \oplus \HH$, where $\HH = \text{span}_\RR\{1,i,j,k\}$ with
  $i^2 = j^2 = k^2 = ijk =-1$. We define a nowhere vanishing vector field $X$ on $S^5$ by
  \[
    X(z,q) := (iz,iq).
  \]
  Then $E_0:=(\RR \cdot X)^\perp$ is a quaternionic line bundle over $S^5$. The map
  \[
    \sigma \colon M \to E_0, \quad (z,q) \mapsto (0,jq)
  \]
  defines a transverse section of $E_0$ with zero locus $L_0:= S^1 \times 0\subset \CC \oplus
  \HH$. The framing of $\nu(L_0)$ induced by $E_0|_{L_0}$ is
  represented by the trivializations
  \[
    \tau_N(z,0) = (0,N),\quad N \in \{1,i,j,k\}.
  \]
  Now, constructing the stable framing for $TL_0$ with respect to the unique $\Spg(1)$-structure
  of $S^5$ gives the Lie group framing of $S^1$, thus  $[L_0,\varphi_E^S]$ represents the generator
  of $\Omega_{1}^{\text{fr}}$ and therefore $\kappa(E_0)=1$.
\end{exmpl}
There is a subtle point in the definition of $\kappa$ with the choice of the $\Spg(1)$-structure
on $M$ . In general the definition of this invariant depends on the $\Spg(1)$-structure but for
some quaternionic line bundles the invariant is independent from that choice.

Fix a $\Spg(1)$-structure of $M$, then, by obstruction theory, any other $\Spg(1)$-structure
is determined by an element of $H^1\left( M;\pi_1\left(\SOg(5)/\Spg(1)\right)
\right)=H^1(M;\ZZ_2)$ which is the obstruction that any other structure is homotopic to the fixed
one.
\begin{prop}\label{P:KappaSpinStructureGeneral}
  Fix a $\Spg(1)$-structure on $M$ and choose another one represented by $\alpha \in H^1(M;\ZZ_2)$.
  Denote by $\kappa^\alpha$ the generalized Kervaire semi-characteristic induced by the
  $\Spg(1)$-structure $\alpha$. If \mbox{$\alpha \smile w_4(E)=0$} then $\kappa(E)=\kappa^\alpha(E)$.
\end{prop}

\begin{proof}
  Let $\sigma \in \Gamma(E)$ be a real transverse section then the zero locus $L$ of $\sigma$
  is the Poincar\'{e} dual to the Euler class $e(E)$ of the underlying vector bundle with
  structure group $\SOg(4)$, cf. \cite[Proposition 12.8]{MR658304}.
  Furthermore we have the well known relation $w_4(E) \equiv e(E_\RR) \mod 2$. We write
  $L=L_1\cup \ldots \cup L_k$ where $L_i$ are connected embedded circles of $M$.

  We have to show that for any $\alpha \in H^1(M;\ZZ_2)$ with $\alpha \smile w_4(E)=0$ the induced
  $\Spg(1)$-structure on $TM|_L$ does not change. This is equivalent to the vanishing of
  $i^\ast(\alpha) \in H^1(L;\ZZ_2)$, where $i \colon L \to M$ is the embedding of $L$ into $M$.
  Set $i_j:=i|_{L_j}\colon L_j \to M$ and
  denote by $\mu_L$ the generator of $H_1(L;\ZZ_2)$ and $\mu_{L_j}$ that of $H_1(L_j;\ZZ_2)$.
  Clearly we have $i_\ast(\mu_L) = \sum_{j=1}^{k} (i_{j})_\ast(\mu_{L_j})\in H_1(M;\ZZ_2)$.

  By Poincar\'{e} duality $i^\ast(\alpha)=0$
  if and only if $i^\ast(\alpha) \frown \mu_L=0$. It follows
  \[
    i^\ast(\alpha)\frown \mu_L = \sum_{j=1}^{k} (i_j)^\ast(\alpha)\frown \mu_{L_j}
  \]
  and the sum is zero if and only if
  \[
    \sum_{j=0}^{k} (i_j)_\ast\left( (i_j)^\ast(\alpha)\frown \mu_{L_j} \right)=0
  \]
  since $(i_j)_\ast \colon H_0(L_j;\ZZ_2) \to H_0(M;\ZZ_2)$ is an isomorphism for all $j$.
  But from
  \[
    (i_j)_\ast\left( (i_j)\ast(\alpha)\frown \mu_{L_j}  \right)=\alpha \frown
  (i_j)_\ast(\mu_{L_j})
  \]
  we deduce that $\sum_{j=0}^{k} (i_j)_\ast\left( (i_j)^\ast(\alpha)\frown \mu_{L_j} \right)
  =\alpha \frown i_\ast(\mu_L)$. Note that $i_\ast(\mu_L)$ is the Poincar\'{e} dual of $w_4(M)$ by
  construction. Furthermore if $[M] \in H_5(M;\ZZ_2)$ is the fundamental class of $M$ we compute
  the Poincar\'{e} dual
  \[
    \left( \alpha \smile w_4(E)  \right)\frown [M]
    =\alpha \frown \left( w_4(E)\frown [M]  \right) = \alpha \frown i_\ast(\mu_L),
  \]
  which proves the proposition.
\end{proof}
\begin{cor}\label{C:H1equalsZero}
  If $w_4(E)=0$ then $\kappa(E)$ is independent of the chosen $\Spg(1)$-structure on $M$. In
  particular this is true if $H^1(M;\ZZ_2)=0$.
\end{cor}
\begin{rem}\label{R:JHomomorphism}
  If $L\subset M$ is an embedded circle with a normal framing $\varphi$, then, using the
  isomorphism $\varepsilon^5 \cong TL \oplus \varepsilon^4$, we may define a map $L \to \SOg(5)$ by
  choosing a trivialization of $TL$ and expressing the framing of $TL\oplus \varepsilon^4$ in a
  basis of the trivialization of $TM|_L \cong \varepsilon^5$. This defines a homotopy class in
  $\pi_1(\SOg(5))$. If $\pi_1(\SOg(5))$ is identified with $\pi_1^S$ through the J-homomorphism we
  see that $[L,\varphi^S]$ agrees with the homotopy class of $L \to \SOg(5)$.
\end{rem}
\begin{lem}\label{L:FramingUnderDiffeomorphism}
  Suppose $M$ and $N$ are oriented, closed spin $5$-manifolds and $\Phi \colon M \to N$ a smooth
  map such that $\Phi|_U \colon U \to V$ is a diffeomorphism, where $U\subset M$, $V \subset N$ are
  open subsets. Let $(L_N,\varphi_N)$ be a normally framed circle in $N$ and set $L_M := \Phi^{-1}(L_N)$,
  $\varphi_M := \Phi^\ast(\varphi_N)$. If $0=[L_N]_2 \in H_1(N;\ZZ_2)$, then $[L_M,\varphi_M^S]$ is
  equal to $[L_N,\varphi_N^S]$.
\end{lem}
\begin{proof}
  Choose a $\Spg(1)$-structure on $M$ and $N$. Let $S_N \colon L_N \to \SOg(5)$ be the map
  obtain from $[L_N,\varphi_N]$ according to Remark \ref{R:JHomomorphism}. Then the corresponding
  map $S_M \colon L_M \to \SOg(5)$ is given as
  \[
    S_M = A \cdot S_N \cdot  A^{-1}
  \]
  where $A \colon L_M \to \SOg(5)$ is defined by expressing a trivialization of $TN|_{L_N}\cong
  \varepsilon^5$ via $D(\Phi^{-1})$ through a trivialization of  $TM|_{L_M}$. Using an
  \emph{Eckmann-Hilton} argument we obtain for the homotopy classes in $\pi_1(\SOg(5))$:
  \[
    [S_M] = [A] + [S_N] + [A^{-1}] =[A] + [S_N] + [A] = [S_N].
  \]
  Since $[L_N]_2=0$ we have also $[L_M]_2=0$ and the classes
  $[L_N,\varphi_N^S]$ and $[L_M,\varphi_M^S]$ do not depend on the $\Spg(1)$-structures.
\end{proof}
\begin{thm}\label{T:kappaIsAsection}
  Let $M$ be a oriented, closed spin $5$-manifold. Then for any $\Spg(1)$-structure on $M$ we have
  \begin{enumerate}[label=(\alph*)]
    \item the generator of $\ZZ_2 \cong \ker\tfrac{p_1}{2}\subset \pi^4(M)$ is given by the
      homotopy class of $\nu \circ p \colon M \to S^4$ where $p \colon M \to S^5$ is a map of
      odd degree and $\nu$ represents the generator of $\pi_5(S^4)$.
    \item if $\Omega_1^{\text{fr}}$ is identified with $\pi_1^S \cong
      \pi_5(S^4) \subset \pi^4(M)$, then $\kappa \colon \pi^4(M) \to \Omega_1^{\text{fr}}$
      is a splitting map for the short exact sequence \eqref{eq:SES}.
  \end{enumerate}
\end{thm}
\begin{proof}
  Clearly the quaternionic line bundle $E:= (\nu \circ p)^\ast(H)$ lies in the kernel of $\Phi$ as
  well as the trivial bundle $\underline{\mathbb H}$. Then part $(a)$ follows if we show that $E$
  is not the trivial bundle.
  Assume first $p \colon M \to S^5$ is a map of degree $1$. Thus there are open sets $U\subset M$
  and $V \subset S^5$ such that $p|_U \colon U \to V$ is a diffeomorphism. Let $[L_0,\varphi_0] \in \pi^4(S^5)$ be the generator and we may assume that $L_0
  \subset V$. Set $L:=p^{-1}(L_0)$ and pull back the framing $\varphi:=p^\ast(\varphi_0)$. The
  element $[L_0,\varphi_0]$ is the framed divisor of $E_0:=\nu^\ast(H)$ and $[L,\varphi]$ that of
  $E:=(\nu\circ p)^\ast(H)$. Since $[L_0]_2 \in H_1(S^5;\ZZ_2)$ is zero we
  obtain from Lemma \ref{L:FramingUnderDiffeomorphism} that $\kappa(E_0)=\kappa(E)$, hence
  $\kappa(E)$ is a generator of $\Omega_1^{\text{fr}} \cong \pi_1^S \cong \pi_5(S^4)$, see
  Example \ref{E:5Sphere}. The same argument shows that $\kappa(\underline{\mathbb H})$ is zero.

  If $\deg p$ is an odd number different from $1$, then there is an odd number of points in the
  preimage of a regular value for which $p$ is a local diffeomorphism around these points.
  Hence for $\kappa(E)$ we would sum - by definition - an odd number of $\kappa(E_0)$, which again gives
  the generator of $\Omega_1^{\text{fr}}$. Same holds for $\kappa(\underline{\mathbb H})$.
   Thus $E$ and
   $\underline{\mathbb H}$ cannot be isomorphic (cf. Corollary \ref{C:IvanricanceUnderIsomorphism}),
   which proves part $(a)$.

  We explain next that $\kappa(E)$ is a homomorphism. For $E,F \in \pi^4(M)$ the corresponding
  framed divisor for $E+F$ is given by disjoint union, cf. Remark \ref{R:GroupStructureOnBun_0}.
  Same holds for the group structure of $\Omega_1^\text{fr}$. Thus by definition we have for the
  stabilized framed divisors
  \[
  \kappa(E+F) = [L_E\cup L_F,(\varphi_E\cup\varphi_F)^S]=[L_E,\varphi_E^S] + [L_F,\varphi_F^S]
  = \kappa(E) + \kappa(F).
  \]

  Now it follows that $\kappa$ splits the short exact sequence \eqref{eq:SES}, since $\kappa$ is
  a homomorphism and $\kappa$ restricted to $\ker\tfrac{p_1}{2}$ is the identity, which follows
  from part $(a)$.
\end{proof}
\begin{cor}\label{C:TheSplitting}
Let $M$ be a closed, oriented spin $5$-manifold. Then for any $\Spg(1)$-structure on $M$ the map
\[
  \pi^4(M) \to H^4(M;\ZZ) \times \Omega_1^{\text{fr}},\quad E\mapsto \left(
  \frac{p_1}{2}(E),\kappa(E)  \right)
\]
is an isomorphism. Thus two quaternionic line bundles $E$ and $E'$ are isomorphic if and only if $E$
  and $E'$ are stably isomorphic and $\kappa(E)=\kappa(E')$.
\end{cor}
Moreover from the proof of Theorem \ref{T:kappaIsAsection} we conclude
\begin{cor}\label{C:Naturality of kappa}
  Suppose $M$ and $N$ are two closed oriented spin $5$-manifolds and $\Phi \colon M \to N$
  a continuous map. If $E \to N$ is a quaternionic line bundle, then for any $\Spg(1)$-structure on $M$
  and $N$ we have
  \[
    \kappa(\Phi^\ast(E))= \deg_2\Phi \cdot \kappa(E),
  \]
  where $\deg_2\Phi$ is the mod $2$ degree of $\Phi$.
\end{cor}
\section{Classification of spin vector bundles of rank $5$}\label{S:Classification of spin vector
bundles}
Let $V \to M$ be a spinnable vector bundles of rank $5$.
Note that for $V$ there is a quaternionic line bundle $E \in \pi^4(M)$ such that
$V \cong E \oplus \varepsilon^1$. Any other quaternionic line bundle with this property is
stably isomorphic to $E$. Hence there are at most two quaternionic line bundles with $V \cong
E \oplus \varepsilon^1$.
\begin{lem}\label{L:w_4(V)=0}
  If $w_4(V)=0$ then there is a unique quaternionic vector bundle $E$ such that $V \cong E \oplus
  \varepsilon^1$.
\end{lem}

\begin{proof}
Suppose $E$ and $E'$ are quaternionic vector bundles with $E \oplus \varepsilon^1 \cong V\cong E'
  \oplus \varepsilon^1$. Note that $w_4(E)=w_4(E')=0$, thus $\kappa(E)$ and $\kappa(E')$ do not
  depend on the $\Spg(1)$-structure of $M$. $E$ and $E'$ are stably isomorphic and we
  may assume that the zero loci, say $L$, of transverse section from both bundles are the same. For
  $\kappa(E)$ and $\kappa(E')$ we have to consider the stable framings
  \[
    \varepsilon^5 = TM|_L = TL \oplus E|_L,\quad
    \varepsilon^5 = TM|_L = TL \oplus E'|_L.
  \]
  Adding on both sides a trivial bundle will no change the class in $\Omega_1^{\text{fr}}$. We
  obtain
  \[
    \varepsilon^6 = \varepsilon^1 \oplus TM|_L = TL \oplus (E \oplus \varepsilon^1)|_L
    \cong TL \oplus V|_L
  \]
  and the same for $E'$, thus $\kappa(E)=\kappa(E')$ and with Corollary \ref{C:TheSplitting} we
  obtain $E=E'$.
\end{proof}
From Lemma \ref{L:w_4(V)=0} we may define for $V$ with $w_4(V)=0$ also a $\ZZ_2$-invariant by
$\kappa(V):=\kappa(E)$, where $E$ is the unique quaternionic line bundle of $V$.

We will explain in the next lines that $\kappa(V)$ is the same invariant as the \emph{generalized
Kervaire semi-characteristic} $k(V)$ defined in \cite{MR1938494}. We give a brief description
of $k(V)$ (cf. \cite{MR0215317,MR1938494}; note that the following lines can be generalized to
higher dimensions): Since $w_4(V)=0$ there are two cross sections
$\sigma_1,\sigma_2$ of $V$ which are linearly independent on the complement of a finite set of
points $\{p_1,\ldots, p_k\}$. We may assume that $p_i$ lies in the interior of a closed simplex
$s$. The vector bundle $V$ restricted to $s$ is the trivial bundle $s \times \RR^5$ and thus
$(\sigma_1(q),\sigma_2(q))$ can be regarded as an orthonormal $2$-frame in $\RR^5$ on
$s\setminus\{p_i\}$, thus a point in the Stiefel manifold $V_{5,2}$. The boundary $\partial s$ of
$s$ is a $4$-sphere and the restriction of $(\sigma_1,\sigma_2)$ on $\partial s$ produces
a map $\partial S \to V_{5,2}$. Its homotopy class, call it $I_{p_i}$, is therefore an element of
$\pi_4(V_{5,2})$. One defines now
\[
  k(V):= \sum_{i=1}^k I_{p_i} \in \pi_4(V_{5,2}) \cong \ZZ_2.
\]

The authors show

\begin{lem}[From Corollary 2.2 in \cite{MR1938494}]\label{L:TangZhangCor}
  $k(V)$ is independent of the choices made above and moreover $V$ admits two linearly independent
  cross sections if and only if $k(V)$ vanishes.
\end{lem}
This invariant coincides with the invariant $k$ of Thomas defined in \cite[p.\, 108]{MR0215317},
which is the $k$-invariant of the Moore-Postnikov tower of a certain fibration. In case $V=TM$ the
invariant is well known as the \emph{Kervaire semi-characteristic of $M$}
\cite{MR0215317,MR0263102} and can be easily
computed by the formula
  \begin{align*}
    k(M) &\equiv \sum_{i}^{} \dim H^{2i}(M;\ZZ_2)\mod 2 \\
    &\equiv \sum_{i}^{} \dim H^{2i}(M;\RR)\mod 2,
  \end{align*}
where the last equality is a consequence of the \emph{Peter-Lusztig-Milnor formula} \cite{MR0246308}.
\begin{prop}\label{P:Connection to generalized Semi Kervaire Charactersitic}
  Let $V$ be an oriented, spinnable vector bundle over $M$ of rank $5$ and such that $w_4(V)=0$.
  Then $\kappa(V)=k(V)$.
\end{prop}

\begin{proof}
  From \cite[Corollary 2.2]{MR1938494} we have $k(V)=0$ if and only if there exists two
  linearly independent sections of $V$, we would like to show $\kappa(V)=0$.
  We will use in the following lines the \emph{singularity approach} of \cite{MR611333}.

  Let $\psi \in \Gamma(V)$ be a nowhere vanishing section such that $V \cong E \oplus \RR \cdot
  \psi$, where $E$ is the unique quaternionic line bundle to $V$. Let $\sigma$ be a transverse
  section of $E$. This data induce a bundle morphism
  \[
    u\colon \varepsilon^2 \to V,\quad u_p(x_1,x_2) := x_1\sigma(p) + x_2\psi(p),\quad p\in M,\,
    (x_1,x_2) \in \RR^2,
  \]
  which in turn defines a section, denoted by $s_u$, in the homomorphism bundle
  $\Hom(\varepsilon^2,V)$. Let $W^1,A^1 \subset \Hom(\varepsilon^2,V)$ be the subfibrations
  where the fibers consist of all linear maps of rank $\geq 1$ and rank equal to $1$ respectively.
  Then $W^1$ is open in $\Hom(\varepsilon^2,V)$ and $A^1$ is a closed smooth submanifold of $W^1$,
  cf. \cite[p.\,17]{MR611333}.
  Of course we have $s_u \in \Gamma\left( W^1  \right)$ and $s_u$ is
  transverse to $A^1$. Moreover it is evident that $s_u^{-1}(A^1) =:L$ is the zero locus of
  $\sigma$. Let $s_0 \in \Gamma(W^1)$ be the morphism induced
  by the two linearly independent sections of $V$, which exist, since $k(V)=0$ is assumed.

  Next, we look at $\widetilde\Hom := \Hom(\varepsilon^2,\text{pr}^\ast(V)) = \text{pr}^\ast
  \left(\Hom(\varepsilon^2,V)\right)$ and consider $\widetilde A^1$ as well as $\widetilde W^1$
  analogously to the definitions above. Then there exists a section $S \in \Gamma(\widetilde W^1)$
  with the properties that $S$ is transverse to $\widetilde A^1$, $S|_{M\times 0} = s_u$
  and $S|_{M\times 1} = s_0$. The set $\Sigma:= S^{-1}(\widetilde A^1)$ is a smooth compact
  surface with $\partial\, \Sigma =L$, hence $\Sigma$ is a null-bordism for $L$. According
  to the definition of $\kappa(V)$ we have to show that the stable framing of $TL$ in the
  definition of $\kappa(V)$
  is induced by a stable framing of $T\Sigma$.

  The normal bundle of $\widetilde A^1$ in $\widetilde{W^1}$ is isomorphic to
  $\Hom(\widetilde\ker,\widetilde \coker)$
  where
  \[
   \widetilde{\ker} := \bigcup_{h \in \widetilde A^1} h \times \ker h,
   \quad
   \widetilde{\coker} := \bigcup_{h \in \widetilde A^1} h \times \coker h,
  \]
  see \cite[equation (1.1)]{MR611333}. This implies $\nu(\Sigma \hookrightarrow M\times I) \cong
  \Hom(\varepsilon^1,\text{pr}^\ast(V)/\im S)|_\Sigma \cong (\text{pr}^\ast(V)/\im S)|_\Sigma$ and
  therefore we have
  \[
    \nu\left( \Sigma \hookrightarrow M\times I  \right) \oplus \RR \cdot S|_\Sigma
    \cong \text{pr}^\ast(V)|_{\Sigma}.
  \]
  Since $\text{pr}^\ast(V)$ has, up to homotopy, a unique framing we obtain, after a choice
  of an $\Spg(1)$-structure on $M$, a stable framing
  of $T\Sigma$ by
  \[
    \varepsilon^7\cong (T(M\times I)\oplus \RR \cdot S)|_\Sigma
    \cong T\Sigma \oplus \nu(\Sigma \hookrightarrow M\times I)\oplus \RR \cdot S|_\Sigma
    \cong T\Sigma \oplus\varepsilon^5.
  \]
 One sees immediately that the restricted stable framing on $L \times 0$ is that which is used
  for the definition of $\kappa(V)$. This shows that $k(V)$ implies $\kappa(V)=0$.

  It remains to show that $k(V)=1$ implies $\kappa(V)=1$. Let $\overline V$ be the bundle stably
  isomorphic to $V$ but $\kappa(\overline V)=1+\kappa(V)$. Then $V$ and $\overline V$ are not isomorphic
  (since the corresponding quaternionic line bundles are not isomorphic),
  hence from \cite[Lemma 3]{MR0224109} we must have $k(\overline V)=0$ and the first part of this
  proof implies $\kappa(\overline V)=0$ thus $\kappa(V)=1$.
\end{proof}
\begin{lem}\label{L:w_4(V) is not 0}
  Suppose $w_4(V)\neq 0$. Then $V$ is uniquely determined by $\tfrac{p_1}{2}(V)$.
\end{lem}

\begin{proof}
  Since $w_4(V)\neq 0$ there is a non-zero $\alpha \in H^1(M;\ZZ_2)$ such that $\alpha\cup
  w_4(V)\neq 0$. Thus there are at least two non-homotopic lifts of the classifying map $M \to B \SOg(5)$
  of $V$. This means there are two non-isomorphic quaternionic line bundles $E$ and $E'$ such that
  $E \oplus\varepsilon^1 \cong V \cong E' \oplus\varepsilon^1$. As mentioned above, there can be
  at most two such quaternionic line bundles. Hence $V$ is completely determined by
  $\tfrac{p_1}{2}(V)$.
\end{proof}

Combining Lemmas \ref{L:w_4(V)=0}, \ref{L:w_4(V) is not 0} and Corollary \ref{C:TheSplitting} we obtain
\begin{thm}\label{T:Classification of V}
Let $M$ be a closed spin $5$-manifold and consider the sets
  \begin{align*}
    W_1 &:= \{ V \in [M,B \SOg(5)] : w_2(V)=w_4(V)=0\},\\
    W_2 &:= \{ V \in [M,B \SOg(5)] : w_2(V)=0,\,w_4(V)\neq 0\}.
  \end{align*}
Then $W_1$ is a group such that the map
\[
  W_1 \to  H^4(M;\ZZ) \oplus\Omega_1^{\text{fr}}, \quad V \mapsto \left( \tfrac{p_1}{2}(V), \kappa(V)  \right)
\]
  is an isomorphism of groups.

  Furthermore if $\dim H^4(M;\ZZ_2)>0$ then every element in $W_2$ is
  uniquely determined by its spin characteristic class $\tfrac{p_1}{2}$ which are in one-to-one
  correspondence to $\rho^{-1}_2(H^4(M,\ZZ_2)\setminus\{0\})$, where $\rho_2$ induced by the mod $2$ reduction
  $\ZZ \to \ZZ_2$. In particular every vector bundle in $W_2$ is determined by its stable class.
\end{thm}
\section{Applications}\label{S:Applications}
We will collect some consequences of the proceedings sections concerning the topology of
$5$-manifolds. Denote with $\beta_i := \dim H^i(M;\ZZ_2)$ and $b_i := \dim H^i(M;\RR)$
the Betti numbers of $M$.

It follows from \ref{SS:GeometricInterpretation} that $M$ is stably parallelizable if
$\tfrac{p_1}{2}(M)=0$, where we define $\tfrac{p_1}{2}(M)$ to be $\tfrac{p_1}{2}(TM)$. Thus
with Proposition \ref{P:Connection to generalized Semi Kervaire Charactersitic} we deduce that $M$
is parallelizable if and only if $\kappa(TM)=k(M)=0$.

\begin{prop}\label{P:ParallelizableSpin5Manifolds}
Let $M$ be a closed, oriented, stably parallelizable spin $5$-manifold. If $\beta_2+\beta_4 \equiv
b_2 +b_4 \equiv 1 \mod 2$, then $M$ is parallelizable, otherwise $TM \cong p^\ast(TS^5)$ for $p
\colon M\to S^5$ a map of odd degree.
\end{prop}

\begin{proof}
  Since $M$ is spin we have $w_4(M)=0$ and by assumption $\tfrac{p_1}{2}(M)=0$. From Theorem
  \ref{T:Classification of V} there are exactly to isomorphism classes of spinnable vector bundles
  of rank $5$ over $M$, which are distinguished by $\kappa$. One of them has to be the trivial
  bundle $\varepsilon^5$ with $\kappa(\varepsilon^5)=0$. Thus from Proposition \ref{P:Connection to
  generalized Semi Kervaire Charactersitic} the tangent bundle is isomorphic to $\varepsilon^5$ if
  and only if $k(M)=0$, i.e. $\beta_2 + \beta_4 \equiv 1 \mod 2$ or $b_2 + b_4 \equiv 1 \mod 2$
  (where the two sums are equal due to the Peter-Lusztig-Milnor formula \cite{MR0246308}).

  If this is not the case, then the non-trivial stably parallel quaternionic line bundle $E$
  is given by $(\nu \circ p)^\ast(H) $, where $p \colon M \to S^5$ is a map of odd degree,
  $\nu \in \pi_5(S^4)$ the generator and $H \to \HH\PP^{1}$ the tautological line bundle,
  see Theorem \ref{T:kappaIsAsection}. Thus from Lemma \ref{L:w_4(V)=0} the vector bundles
  $TM$ and $(\nu\circ p)^{\ast}(H)\oplus \varepsilon^1 \cong p^\ast(TS^5)$ have to be isomorphic.
\end{proof}
\begin{rem}\label{R:BredonKosinski}
  Proposition \ref{P:ParallelizableSpin5Manifolds} is partly known:
In \cite{MR0200937} the authors show that a stably parallelizable manifold is parallelizable
  if and only if $k(M)=0$.
\end{rem}

An immediate consequence of Corollary \ref{C:TheSplitting} is
\begin{cor}\label{C:SimplyConnected}
  If $H_1(M;\ZZ)=0$ then $M$ is stably parallelizable. There are exactly two isomorphism
  classes of spinnable rank $5$ vector bundles over $M$ and $M$ has trivial tangent bundle
  if and only if $\beta_2 \equiv b_2 \equiv 1 \mod 2$. Otherwise $TM \cong p^\ast(TS^5)$.
\end{cor}


From the naturality property of $\kappa$ we may deduce properties for coverings of spin
$5$-manifolds.

\begin{prop}\label{P:Coverings}
Suppose $M$ is a closed, oriented spin $5$-manifold with finite fundamental group. Furthermore
assume that the universal cover $\pi \colon\widetilde M \to M$ is also spin. Denote by
  $\#\pi$ the cardinality of the fibers of $\pi$. Then
  \begin{enumerate}[label=(\alph*)]
    \item If $\#\pi \equiv 0 \mod 2$ then $\widetilde M$ has trivial tangent bundle . Thus
      $\widetilde\beta_2 \equiv \widetilde b_2 \equiv 1\mod 2$, where $\widetilde\beta_2,\widetilde
      b_2$ are the second Betti numbers of $\widetilde M$ with $\ZZ_2$ and $\RR$ coefficients
      respectively.
    \item If $\#\pi \equiv 1 \mod 2$ then $\widetilde M$ has trivial tangent bundle if and only
      if $\beta_2 + \beta_4 \equiv  b_2 +b_2 \equiv 1 \mod 2$.
  \end{enumerate}
\end{prop}

\begin{proof}
  Since $\widetilde{TM} = \pi^\ast(TM)$ we have $\kappa(\widetilde{TM}) =\deg_2\pi \cdot  \kappa(TM)
  =\#\pi \cdot  \kappa(TM)$, see Corollary \ref{C:Naturality of kappa}. The proposition follows
  from Proposition \ref{P:Connection to generalized Semi Kervaire Charactersitic} and the formula
  for the Kervaire semi-characteristic.
\end{proof}

\begin{prop}\label{P:Fundamentalgroup iso to z2}
  Let $M$ be a closed spin $5$-manifold with $H_1(M)\cong \ZZ_2$. Then $M$ is stably parallelizable
  and $M$ is parallelizable if and only if $\beta_2 \equiv b_2 \equiv 0 \mod 2$.
\end{prop}

\begin{proof}
  Since $w_4(M) \equiv \tfrac{p_1}{2}(M) \mod 2$ and $w_4(M)=0$ since $M$ is spin, we obtain
  $\tfrac{p_1}{2}(M)=0$ because of $H^4(M;\ZZ) \cong H_1(M) \cong \ZZ_2$.

\end{proof}
\bibliographystyle{acm}
\bibliography{vb5manifolds}
\end{document}